\documentclass{article}

\usepackage{amssymb,amsfonts,amsthm, amscd}
\usepackage{amsmath}
\usepackage{hyperref}

\theoremstyle{plain}
\numberwithin{equation}{section}
\newtheorem{theorem}{Theorem}[section]
\newtheorem{lemma}{Lemma}[section]

\theoremstyle{remark}

\theoremstyle{definition}
\newtheorem{definition}{Definition}

\begin{document}

\begin{center}
	\textbf{\Large{Gradient of eigenvalues of Dirac operators \\ and its applications}}
	
	\vskip 7pt
	
	Tigran Harutyunyan and Yuri Ashrafyan
	
	\vskip 7pt
	
	Department of Mathematics and Mechanics, Yerevan State University, Yerevan, Armenia
\end{center}

\begin{abstract}
For Dirac operators, which have discrete spectra, the concept of eigenvalues' gradient is given and  formulae for this gradients are obtained in terms of normalized eigenfunctions. 
It's shown how the gradient is being used to describe isospectral operators or when finite number of spectral data is changed.

\end{abstract}

\parindent=1cm

\section{\large{Introduction. Gradient of eigenvalues.}}\label{sec1}
Let $E$ is two dimensional identical matrix, and 
$\sigma_1=\left(
                                                             \begin{array}{cc}
                                                               0 & i \\
                                                               -i & 0 \\
                                                             \end{array}
                                                           \right)
$,
$\sigma_2=\left(
                                                             \begin{array}{cc}
                                                               1 & 0 \\
                                                               0 & -1 \\
                                                             \end{array}
                                                           \right)
$,
$\sigma_3=\left(
                                                             \begin{array}{cc}
                                                               0 & 1 \\
                                                               1 & 0 \\
                                                             \end{array}
                                                           \right)
$
are well-known Pauli matrices, which have properties $\sigma_k^2 = E$, $\sigma_k^* = \sigma_k$ (self-adjointness) and $\sigma_k \sigma_j = - \sigma_j \sigma_k$ (anti-commutativity), when $k \neq j$, for $k,j = 1, 2, 3$.

Let $p$ and $q$ are real-valued, summable on $[0,\pi]$ functions, i.e. $p, q \in L^1_{\mathbb{R}}[0, \pi]$.
By $L(p, q, \alpha, \beta) = L(\Omega, \alpha, \beta)$ we denote the boundary-value problem for canonical Dirac system (see \cite{GaLe,LeSa,Ma1,GaDz}):
\begin{gather}
\quad \ell y\equiv \Big\{ B \cfrac{d}{dx} + \Omega(x) \Big\} y =\lambda y,
\quad x \in (0, \pi),
\quad y=\left(
                                                                     \begin{array}{c}
                                                                       y_1 \\
                                                                       y_2 \\
                                                                     \end{array}
                                                                   \right)
,\quad \lambda\in \mathbb{C},\label{eq1.1}\\
y_1(0)\cos\alpha+y_2(0)\sin\alpha=0,\quad \alpha\in \Big(-\cfrac{\pi}{2}, \cfrac{\pi}{2} \Big],\label{eq1.2}\\
y_1(\pi)\cos\beta+y_2(\pi)\sin\beta=0,\quad \beta\in \Big(-\cfrac{\pi}{2}, \cfrac{\pi}{2} \Big],\label{eq1.3}
\end{gather}
where
$B = \cfrac{1}{i} \sigma_1 
= \left(
     \begin{array}{cc}
       0 & 1 \\
       -1 & 0 \\
     \end{array}
   \right),
$
$ \ \Omega(x) = \sigma_2 p(x) + \sigma_3 q(x) 
= \left(
             \begin{array}{cc}
               p(x) & q(x) \\
               q(x) & -p(x) \\
             \end{array}
           \right).
$

By the same $L(p, q, \alpha, \beta)$ we also denote a self-adjoint operator, generated by differential expression $\ell$ in Hilbert space of two component vector-functions $L^2([0, \pi]; {\mathbb{C}}^2)$ on the domain
\begin{equation}\label{eq1.4}
\begin{aligned}
D =  \Big\{ y=\left(
\begin{array}{c}
y_1 \\
y_2 \\
\end{array}
\right)
; y_k \in AC [ 0, \pi ], (\ell y)_k \in L^2 [ 0, \pi ], k = 1, 2;\\
 y_1(0) \cos \alpha + y_2(0) \sin \alpha = 0, \ y_1(\pi) \cos \beta + y_2(\pi) \sin \beta = 0 \Big\}
\end{aligned}
\end{equation}
where $AC [ 0, \pi ]$ is the set of absolutely continuous functions on $[0, \pi]$ (see, e.g. \cite{LeSa,Na}).
The scalar product in $L^2([a,b]; \mathbb{C}^2)$ we denote by $(f,g) = \int_{a}^{b} \langle f,g \rangle dx = \int_{a}^{b} [f_1(x) \bar{g}_1(x) + f_2(x) \bar{g}_2(x)] dx$.
It is well known (see \cite{GaDz,HaAz,AlHrMy}) that under these conditions the spectra of the operator $L(p, q, \alpha)$ is purely discrete and consists of simple, real eigenvalues, which we denote by $\lambda_n= \lambda_n (p, q, \alpha, \beta)= \lambda_n (\Omega, \alpha, \beta)$, $n \in \mathbb{Z}$, to emphasize the dependence of $\lambda_n$ on quantities $p, q$ and $\alpha, \beta$.
It is also well known (see, e.g. \cite{GaDz,HaAz,AlHrMy}) that the eigenvalues form a sequence, unbounded below as well as above.
So we will enumerate it as $\lambda_k < \lambda_{k+1}, k \in \mathbb{Z}$, $\lambda_k>0$, when $k>0$ and $\lambda_k<0$, when $k<0$, and the nearest to zero eigenvalue we will denote by $\lambda_0$.
If there are two nearest to zero eigenvalue, then by $\lambda_0$ we will denote the negative one.
With this enumeration it is proved (see \cite{GaDz,HaAz,AlHrMy}), that the eigenvalues have the asymptotics:

\begin{equation}{\label{eq1.5}}
\lambda_n(\Omega, \alpha, \beta)= n - \cfrac{\beta - \alpha}{\pi} + r_n, \quad r_n = o(1), \quad n \rightarrow \pm \infty.
\end{equation}

Let $y(x, \lambda) = \varphi(x,\lambda,\alpha, \Omega)$ and $y (x, \lambda) = \psi(x,\lambda, \beta, \Omega)$ are the solutions of the Cauchy problems

\begin{equation}\label{eq1.6}
\Bigg \{ \begin{array}{c}
\ell y = \lambda y \\
 y (0, \lambda) = \left(
\begin{array}{c}
\sin\alpha \\
-\cos\alpha \\
\end{array}
\right)
\end{array},
\end{equation}

\begin{equation}\label{eq1.7}
\Bigg \{ \begin{array}{c}
\ell y = \lambda y \\
y (\pi, \lambda) = \left(
\begin{array}{c}
\sin\beta \\
-\cos\beta \\
\end{array}
\right)
\end{array},
\end{equation}
respectively.
Since the differential expression $\ell$ is self-adjoint, the components $\varphi_1(x,\lambda)$, $\varphi_2(x,\lambda)$ and $\psi_1(x,\lambda)$, $\psi_2(x,\lambda)$ of the vector-functions $\varphi(x,\lambda)$ and $\psi(x,\lambda)$ can be chosen real-valued for real $\lambda$.
It is easy to see, that $\varphi_n(x,\Omega)=\varphi(x,\lambda_n,\alpha, \Omega)$ and $\psi_n(x,\Omega)=\psi(x,\lambda_n,\beta, \Omega)$ are the eigenfunctions, corresponding to the eigenvalue $\lambda_n$.
By $a_n=a_n(\Omega,\alpha, \beta)$ and $b_n=b_n(\Omega,\alpha, \beta)$ we denote the squares of the $L^2$-norm of the eigenfunctions $\varphi_n(x,\Omega)$ and $\psi_n(x,\Omega)$:

\[
a_n=\|\varphi_n\|^2 = \displaystyle \int_0^{\pi} |\varphi_n (x, \Omega)|^2 dx,\quad    n \in \mathbb{Z},
\]
\[
b_n=\|\psi_n\|^2 = \displaystyle \int_0^{\pi} |\psi_n (x, \Omega)|^2 dx,\quad    n \in \mathbb{Z}.
\]
The numbers $a_n$ and $b_n$ are called norming constants.
By $h_n(x,\Omega)$ we will denote normalized eigenfunctions (i.e. $\|h_n(x)\| = 1$) of operator $L(\Omega, \alpha, \beta)$:
\[
h_n(x) = h_n(x, \Omega) = \cfrac{\varphi_n(x, \Omega)}{\sqrt{a_n(\Omega,\alpha)}},
\]
and it can be taken also as
\[
\hat{h}_n(x) = \hat{h}_n(x, \Omega) = \cfrac{\psi_n(x, \Omega)}{\sqrt{b_n(\Omega,\beta)}}.
\]
It is easy to see, that $|h_n(0)|^2 = \cfrac{1}{a_n}$ and $|\hat{h}_n(\pi)|^2 = \cfrac{1}{b_n}$.
Having a goal to describe the dependence of $\lambda_n$ on quantities $p, q$ and $\alpha, \beta$ more precisely, we input a concept of eigenvalues' gradient, by the following formula (compare with \cite{IsTr})

\begin{equation}\label{eq1.8}
grad \lambda_n = \left( \cfrac{\partial \lambda_n}{\partial \alpha}, \cfrac{\partial \lambda_n}{\partial \beta},
\cfrac{\partial \lambda_n}{\partial p(x)}, \cfrac{\partial \lambda_n}{\partial q(x)} \right).
\end{equation}

\begin{definition}\label{def1.1}
	Let $g$ is defined on $\left( a, b \right)$, where $-\infty\le a < b \le \infty$.
	The derivative of function $f$ with respect to function $g$ is called a function $\cfrac{\partial f}{\partial g(x)}$, which satisfies the equation
	\[
	\cfrac{d}{d \epsilon}f(g + \epsilon v)\Big|_{\epsilon = 0} = \displaystyle \int_{a}^{b} \cfrac{\partial f}{\partial g(x)}v(x) dx,
	\]
	for all $v \in L^2_{\mathbb{R}}\left( a, b \right)$.
\end{definition}

We want to express the components of the eigenvalues' gradient by normalized eigenfunctions of $L(p,q,\alpha, \beta)$ problem.

\begin{theorem}\label{thm1.1}
  Let $\lambda_n$ and $h_n$ are eigenvalues and normalized eigenfunctions of the $L(p,q,\alpha, \beta)$ problem correspondingly.
  Then there hold the relations:
\begin{gather*}
  \cfrac{\partial \lambda_n(\alpha, \beta, p, q)}{\partial \alpha} = - |h_n(0)|^2, \\
  \cfrac{\partial \lambda_n(\alpha, \beta, p, q)}{\partial \beta} = |h_n(\pi)|^2, \\
  \cfrac{\partial \lambda_n(\alpha, \beta, p, q)}{\partial p(x)} = |h_{n_1}(x)|^2 - |h_{n_2}(x)|^2, \\
  \cfrac{\partial \lambda_n(\alpha, \beta, p, q)}{\partial q(x)} = 2 h_{n_1}(x) \cdot h_{n_2}(x).
\end{gather*}
\end{theorem}

\begin{proof}
  Let $h_n$ is eigenfunction of problem $L(p,q,\alpha, \beta)$, and $\tilde{h}_n$ is eigenfunction of problem $L(p,q,\alpha + \Delta \alpha, \beta)$.
  Then

\begin{equation}\label{eq1.9}
	\ell h_n \equiv B h'_n(x)  + \Omega(x) h_n(x) \equiv \lambda_n(\alpha) h_n(x),
\end{equation} 
\[
h_{n_1}(0)\cos\alpha + h_{n_2}(0)\sin\alpha = 0,
\]
\[
h_{n_1}(\pi)\cos\beta + h_{n_2}(\pi)\sin\beta = 0.
\]

\begin{equation}\label{eq1.10}
\ell \tilde{h}_n \equiv B \tilde{h}'_n(x)  + \Omega(x) h_n(x) \equiv \lambda_n(\alpha + \Delta \alpha) \tilde{h}_n(x),
\end{equation}
\[
\tilde{h}_{n_1}(0)\cos (\alpha + \Delta \alpha) + \tilde{h}_{n_2}(0)\sin (\alpha + \Delta \alpha) = 0,
\]
\[
\tilde{h}_{n_1}(\pi)\cos\beta + \tilde{h}_{n_2}(\pi)\sin\beta = 0.
\]

Multiply \eqref{eq1.9} by $\tilde{h}_n(x)$ scalarly from the right, and \eqref{eq1.10} by $h_n(x)$ from the left.
Taking into account the self-adjointness of $\Omega(x)$ $\left( (h_n, \Omega \tilde{h}_n) = (\Omega h_n, \tilde{h}_n)  \right) $, we obtain

\begin{gather*}
  \Big( B h'_n, \tilde{h}_n \Big) + \Big( \Omega h_n, \tilde{h}_n \Big) = \lambda_n(\alpha) \Big( h_n, \tilde{h}_n \Big),\\
  \Big( h_n, B \tilde{h}'_n \Big) + \Big( \Omega h_n, \tilde{h}_n \Big) = \lambda_n(\alpha + \Delta \alpha) \Big( h_n, \tilde{h}_n \Big).
\end{gather*}
Subtracting from the second equation the first equation, we will get
\[
\int_{0}^{\pi} \Big \langle \left( \begin{array}{c}
                                                 h_{n_1} \\
                                                 h_{n_2}
                                         \end{array} \right),
                                         \left( \begin{array}{c}
                                                 \tilde{h}'_{n_2} \\
                                                 -\tilde{h}'_{n_1}
                                         \end{array} \right)
                                         \Big \rangle
                                         dx
                                         -
\int_{0}^{\pi} \Big \langle \left( \begin{array}{c}
                                                 h'_{n_2} \\
                                                 -h'_{n_1}
                                         \end{array} \right),
                                         \left( \begin{array}{c}
                                                 \tilde{h}_{n_1} \\
                                                 \tilde{h}_{n_2}
                                         \end{array} \right)
                                         \Big \rangle
                                         dx
                                         =
\]
\begin{equation}\label{eq1.11}
\qquad \qquad \qquad \qquad \qquad = \left[\lambda_n(\alpha + \Delta \alpha) - \lambda_n(\alpha) \right] \Big( h_n, \tilde{h}_n \Big).
\end{equation}
Taking into account, that in case of real potentials the components of the solutions can be taken real, thus the left side of the latter equation can be written as
\begin{gather*}
  \int_{0}^{\pi} \Big[ h_{n_1}(x) \tilde{h}'_{n_2}(x) - h_{n_2}(x) \tilde{h}'_{n_1}(x) - h'_{n_2}(x) \tilde{h}_{n_1}(x) + h'_{n_1}(x) \tilde{h}_{n_2}(x) \Big] dx = \\
 = \int_{0}^{\pi} \cfrac{d}{dx} \Big[ h_{n_1}(x) \tilde{h}_{n_2}(x) - h_{n_2}(x) \tilde{h}_{n_1}(x) \Big] dx = \\
 = h_{n_1}(\pi) \tilde{h}_{n_2}(\pi) - h_{n_2}(\pi) \tilde{h}_{n_1}(\pi) - h_{n_1}(0) \tilde{h}_{n_2}(0) + h_{n_2}(0) \tilde{h}_{n_1}(0).
\end{gather*}

Since $h_n(x) = \cfrac{\varphi_n(x, \alpha)}{\sqrt{a_n(\alpha)}}$ \ and \
$\tilde{h}_n(x) = \cfrac{\varphi_n(x, \alpha + \Delta \alpha)}{\sqrt{a_n(\alpha + \Delta \alpha)}}$,
then $h_n(0) = \cfrac{1}{\sqrt{a_n(\alpha)}}\left( \begin{array}{c}
                                             \sin \alpha \\
                                             -\cos \alpha
                                           \end{array}\right)$
\ and \
$\tilde{h}_n(0) = \cfrac{1}{\sqrt{a_n(\alpha + \Delta \alpha)}} \left(\begin{array}{c}
                                             \sin (\alpha + \Delta \alpha) \\
                                             -\cos (\alpha + \Delta \alpha)
                                           \end{array}\right)$.
Thus the equation \eqref{eq1.11} can be rewritten as follows
\[
-\cfrac{1}{\sqrt{a_n(\alpha) a_n(\alpha +\Delta\alpha)}} \sin\Delta\alpha =
\left[\lambda_n(\alpha + \Delta\alpha) - \lambda_n(\alpha) \right] \Big( h_n, \tilde{h}_n \Big).
\]
From the latter, when $\Delta\alpha \rightarrow 0$, we obtain
\begin{equation}\label{eq1.12}
  \cfrac{\partial \lambda_n}{\partial \alpha} = -\cfrac{1}{a_n} = -|h_n(0)|^2.
\end{equation}

Similarly we obtain
\begin{equation}\label{eq1.13}
  \cfrac{\partial \lambda_n}{\partial \beta} = \cfrac{1}{b_n} = |h_n(\pi)|^2.
\end{equation}

To obtain the equality $ \cfrac{\partial \lambda_n}{\partial p(x)} = |h_{n_1}(x)|^2 - |h_{n_2}(x)|^2$, we write \eqref{eq1.9} in the form
\begin{equation}\label{eq1.14}
  B h'_n(x)  + \left( \sigma_2 p(x) + \sigma_3 q(x) \right) h_n(x) \equiv \lambda_n(p) h_n(x)
\end{equation}
and for \eqref{eq1.10} in the form
\begin{equation}\label{eq1.15}
  B \tilde{h}'_n(x)  + \left( \sigma_2 \left[ p(x) + \epsilon v(x) \right] + \sigma_3 q(x) \right) \tilde{h}_n(x) \equiv \lambda_n(p+\epsilon v) \tilde{h}_n(x),
\end{equation}
where $\tilde{h}_n$ is normalized eigenfunction of the $L(p+\epsilon v, q, \alpha, \beta)$ problem.
Multiply \eqref{eq1.14} by $\tilde{h}_n(x)$ scalarly from the right, and \eqref{eq1.15} by $h_n(x)$ from the left.
Taking into account, that $h_n$ and $\tilde{h}_n$ satisfy to the same boundary conditions, subtract  equality \eqref{eq1.14} from  \eqref{eq1.15}, we obtain
\[
\Big( h_n, \sigma_2 \left[ p(x) + \epsilon v(x) \right] \tilde{h}_n \Big) - \Big( \sigma_2 p(x) h_n, \tilde{h}_n  \Big) =
\left[\lambda_n(p+\epsilon v) - \lambda_n(p) \right] \Big( h_n, \tilde{h}_n \Big).
\]
Form the latter it follows
\[
\cfrac{\lambda_n(p+\epsilon v) - \lambda_n(p)}{\epsilon} \Big( h_n, \tilde{h}_n \Big) = \int_{0}^{\pi} \Big( h_{n_1}(x) \tilde{h}_{n_1}(x) - h_{n_2}(x) \tilde{h}_{n_2}(x) \Big) v(x) dx.
\]
Tending $\epsilon \rightarrow 0$, using the fact, that $\tilde{h}_n \rightarrow h_n$, when  $\epsilon \rightarrow 0$ and the definition \ref{def1.1}, we obtain $ \cfrac{\partial \lambda_n}{\partial p(x)} = |h_{n_1}(x)|^2 - |h_{n_2}(x)|^2$.

Similarly we can obtain the equality $\cfrac{\partial \lambda_n}{\partial q(x)} = 2 h_{n_1}(x) \cdot h_{n_2}(x)$.

Theorem \ref{thm1.1} is proved.
\end{proof}

Let us consider also canonical Dirac system on half axis.
Let $p$ and $q$ are real-valued, local summable on $(0,\infty)$ functions, i.e. $p, q \in L^1_{\mathbb{R},loc}(0, \infty)$.
For $\alpha \in \big( -\cfrac{\pi}{2}, \cfrac{\pi}{2} \big]$, by $L(p, q, \alpha)$ we denote the self-adjoint operator, generated by differential expression $\ell$ \ (see \eqref{eq1.1}) in Hilbert space
of two component vector-functions $L^2(( 0, \infty); {\mathbb{C}}^2)$ on the domain
\begin{equation*}
\begin{aligned}
D_{\alpha} =  \Big\{ y=\left(
\begin{array}{c}
y_1 \\
y_2 \\
\end{array}
\right)
; y_k \in L^2 (0, \infty) \cap AC (0, \infty);\\
(\ell y)_k \in L^2 ( 0, \infty), k = 1, 2; \ y_1(0) \cos \alpha + y_2(0) \sin \alpha = 0 \Big\}
\end{aligned}
\end{equation*}
where $AC (0, \infty)$ is the set of functions, which are absolutely continuous on each finite segment
$[a, b] \subset (0, \infty), 0 < a < b < \infty$.
We assume, that the spectra of this operator is pure discrete (see, e.g. \cite{Mar, AshHar1}), and consists of simple eigenvalues, which we denote by $\lambda_n (p, q, \alpha)$, $n \in \mathbb{Z}$.

Let $y = \varphi(x, \lambda, \alpha, \Omega)$ is the same as in the case of finite interval, i.e. $\varphi$ is the solution of Cauchy problem \eqref{eq1.6}. 
Then $\varphi_n(x)=\varphi(x, \lambda_n)$ are the eigenfunctions, $a_n = \int_{0}^{\infty} |\varphi_n(x,\Omega)|^2 dx$, $n \in \mathbb{Z}$, are the norming constants, and $h_n(x) = h_n(x, \Omega, \lambda_n) = \cfrac{\varphi_n(x)}{\sqrt{a_n}}$ are the normalized eigenfunctions. 
In this case the gradient is defined as 
\begin{equation*}
	grad \lambda_n = \left( \cfrac{\partial \lambda_n}{\partial \alpha},\cfrac{\partial \lambda_n}{\partial p(x)}, \cfrac{\partial \lambda_n}{\partial q(x)} \right),
\end{equation*}
and in Definition \ref{def1.1} we take $a = 0, \ b = \infty$.

\begin{theorem}\label{thm1.2}
	Let $\lambda_n$ and $h_n$ are eigenvalues and normalized eigenfunctions of the $L(p,q,\alpha)$ problem correspondingly.
	Then there hold the relations:
	\begin{gather*}
	\cfrac{\partial \lambda_n(\alpha, p, q)}{\partial \alpha} = - |h_n(0)|^2, \\
	\cfrac{\partial \lambda_n(\alpha, p, q)}{\partial p(x)} = |h_{n_1}(x)|^2 - |h_{n_2}(x)|^2, \\
	\cfrac{\partial \lambda_n(\alpha, p, q)}{\partial q(x)} = 2 h_{n_1}(x) \cdot h_{n_2}(x).
	\end{gather*}
\end{theorem}

\begin{proof}
	In case of real potentials the components of the solutions can be taken real.
	Since the eigenfunctions $h_n$ and $\tilde{h}_n$ are from $L^2(0, \infty)$, we can infer that the scalar products $\langle h_n, \tilde{h}_n\rangle$ are from $L^1(0, \infty)$ and, hence, are tending to $0$ on some
	$\{ x_k; \ x_k \rightarrow \infty, \ k \rightarrow \infty, \}$ sequence.
	Taking the latter first two formulae can be proved in the similar way as in theorem \ref{thm1.1}. 
	Thus here we will prove the third formula.
	
	Write the equation \eqref{eq1.9} in the following form
	\begin{equation}\label{eq1.16}
	B h'_n(x)  + \left( \sigma_2 p(x) + \sigma_3 q(x) \right) h_n(x) \equiv \lambda_n(p) h_n(x)
	\end{equation}
	and \eqref{eq1.10} in the form
	\begin{equation}\label{eq1.17}
	B \tilde{h}'_n(x)  + \left( \sigma_2 p(x) + \sigma_3 \left[ q(x) + \epsilon v(x) \right] \right) \tilde{h}_n(x) \equiv \lambda_n(p+\epsilon v) \tilde{h}_n(x),
	\end{equation}
	where $\tilde{h}_n$ is normalized eigenfunction of the $L(p, q + \epsilon v, \alpha, \beta)$ problem.
	Multiplying \eqref{eq1.16} scalarly by $\tilde{h}_n(x)$ from the right, and \eqref{eq1.17} by $h_n(x)$ from the left.
	Taking into account, that $h_n$ and $\tilde{h}_n$ satisfy to the same boundary conditions, subtract  equality \eqref{eq1.16} from  \eqref{eq1.17}, we obtain
	\[
	\Big( h_n, \sigma_3 \left[ q(x) + \epsilon v(x) \right] \tilde{h}_n \Big) - \Big( \sigma_3 q(x) h_n, \tilde{h}_n  \Big) =
	\left[\lambda_n(q+\epsilon v) - \lambda_n(q) \right] \Big( h_n, \tilde{h}_n \Big).
	\] 
	From the latter equation we have
	\begin{equation}\label{eq1.18}
	\begin{aligned}
	\int_{0}^{\infty} \Big( h_{n_1}(x) \bar{\tilde{h}}_{n_2}(x) + h_{n_2}(x) \bar{\tilde{h}}_{n_1}(x) \Big) \epsilon v(x)dx =
	\\
	= \left[\lambda_n(q + \epsilon v) - \lambda_n(q) \right] \Big( h_n, \tilde{h}_n \Big).
	\end{aligned}
	\end{equation}
	And from the equation \eqref{eq1.18} it follows
	\[
	\cfrac{\lambda_n(q+\epsilon v) - \lambda_n(q)}{\epsilon} \Big( h_n, \tilde{h}_n \Big) = \int_{0}^{\infty} \Big( h_{n_1} \bar{\tilde{h}}_{n_2} + h_{n_2} \bar{\tilde{h}}_{n_1} \Big) v(x) dx.
	\]
	Tending $\epsilon \rightarrow 0$, using the fact, that $\tilde{h}_n \rightarrow h_n$, when  $\epsilon \rightarrow 0$ and the definition \ref{def1.1}, we obtain $ \cfrac{\partial \lambda_n}{\partial q(x)} = 2 h_{n_1}(x) h_{n_2}(x)$.
	
	Theorem \ref{thm1.2} is proved.
\end{proof}

It is well-known, that the inverse problem of reconstruction of operator $L(p,q,\alpha, \beta)$ by spectral function (in our case by eigenvalues $\{\lambda_n\}_{n \in \mathbb{Z}}$ and norming constants $\{a_n\}_{n \in \mathbb{Z}}$) can not be solved uniquely, if we permit parameters $\alpha$ and $\beta$ to be arbitrary (see \cite{GaLe}). 
But if we fix one of them, then the inverse problem can be solved uniquely (see \cite{GaLe,AlHrMy,Ha2,Wa}). 
Therefore, usually is considered the problem $L(p,q,\alpha, 0)$ (see \cite{GaDz,AlHrMy,Ha2, AshHar2}).

It is also well-known, that for regular Dirac operators (the operators on finite interval with summable coefficients), we can not add or diminish the eigenvalues (because of obligatory asymptotics \eqref{eq1.5}), staying in the class of summable coefficients, but we can change the norming constants and describe the isospectral Dirac operators (see \cite{Ha2, AshHar2}). 

The applications of eigenvalues' gradient of describing operators, which isospectral with fixed operator $L(p, q, \alpha, 0)$ is given in section \ref{sec2}. 
On the other hand, if we consider Dirac operator on half axis $(0, \infty)$ (which has pure discrete spectra), we can add or diminish arbitrary finite number of eigenvalues or change norming constants, since in this case there are not obligatory asymptotics (see, e.g. \cite{AshHar1}).
The applications of eigenvalues' gradient in this case is given in section \ref{sec3}.

\section{\large{Isospectrality on finite interval.}}\label{sec2}

Let us consider the boundary-value problem $L(p, q, \alpha, 0)$ on $[0,\pi]$.
From the eigenvalues' asymptotics \eqref{eq1.5} it follows:

\begin{equation}{\label{eq2.1}}
\lambda_n(\Omega, \alpha, 0)= n - \cfrac{\alpha}{\pi} + r_n, \quad r_n = o(1), \quad n \rightarrow \pm \infty.
\end{equation}

It is known (see \cite{GaDz,HaAz}) that in the case of $\Omega \in L^2_{\mathbb{R}}[0, \pi]$ the norming constants have an asymptotic form:

\begin{equation}\label{eq2.2}
a_n (\Omega) = \pi + c_n, \quad \sum_{n=-\infty}^{\infty} c_n^2 < \infty.
\end{equation}

\begin{definition}
Two Dirac operators $L(\Omega, \alpha, 0)$ and $L(\tilde{\Omega}, \tilde{\alpha}, 0)$ are said to be isospectral,
if $\lambda_n(\Omega, \alpha, 0) = \lambda_n(\tilde{\Omega}, \tilde{\alpha}, 0)$, for every $n \in \mathbb{Z}$.
\end{definition}

\begin{lemma}\label{lem2.1}
Let $\Omega, \tilde{\Omega} \in L^1_{\mathbb{R}}[0, \pi] $ and
the operators $L(\Omega, \alpha, 0)$ and $L(\tilde{\Omega}, \tilde{\alpha}, 0)$ are isospectral.
Then $\tilde{\alpha} = \alpha$.
\end{lemma}

\begin{proof}
The proof follows from the asymptotics \eqref{eq2.1}:

\[
\cfrac{\alpha}{\pi} = \lim_{n \rightarrow \infty} (n - \lambda_n(\Omega, \alpha, 0)) =
 \lim_{n \rightarrow \infty} (n - \lambda_n(\tilde{\Omega}, \tilde{\alpha}, 0)) = \cfrac{\tilde{\alpha}}{\pi}.
\]
\end{proof}

So, instead of isospectral operators $L(\Omega, \alpha, 0)$ and $L(\tilde{\Omega}, \tilde{\alpha}, 0)$, we can talk about "isospectral potentials"
$\Omega$ and $\tilde{\Omega}$.

Let us fix some $\Omega \in L^2_{\mathbb{R}}[0, \pi] $ and consider the set of all canonical potentials
$\tilde{\Omega} = \left(
                    \begin{array}{cc}
                      \tilde{p} & \tilde{q} \\
                      \tilde{q} & -\tilde{p} \\
                    \end{array}
                  \right)
$,
with the same spectra as $\Omega$:

\[
M^2(\Omega) = \{ \tilde{\Omega} \in L^2_{\mathbb{R}}[0, \pi]:
\lambda_n(\tilde{\Omega}, \tilde{\alpha}, 0) = \lambda_n(\Omega, \alpha, 0), n \in \mathbb{Z} \}.
\]

Our main goal is to give the description of the set $M^2(\Omega)$ in terms of eigenvalues' gradients.
Note, that the problem of description of isospectral Sturm-Liouville operators was solved in \cite{IsTr,IsMcTr,DaTr,PoTr}.

For Dirac operators the description of $M^2(\Omega)$ is given in \cite{Ha2}.
This description has a "recurrent" form, i.e. at the first in \cite{Ha2} is given the description of a family of isospectral potentials $\Omega (x, t), t \in \mathbb{R}$, for which only one norming constant $a_m (\Omega (\cdot, t))$ different from $a_m (\Omega)$ (namely, $a_m (\Omega (\cdot, t)) = a_m (\Omega) e^{-t}$), while the others are equal, i.e. $a_m (\Omega (\cdot, t)) = a_m (\Omega)$, when $n \neq m$.

\begin{theorem}\label{thm2.1}{\cite{Ha2}.}
Let $t \in \mathbb{R}$, $\alpha \in \Big( - \frac{\pi}{2}, \frac{\pi}{2} \Big]$. Then
\footnote{Here * is a sign of transponation, e.g.
$h_m^{*} =
 \left(
    \begin{array}{c}
      h_{m_1} \\
      h_{m_2} \\
    \end{array}
  \right)^{*}
= ( h_{m_1}, h_{m_2} )
$}

\[
\Omega(x,t) = \Omega(x) + \cfrac{e^t -1}{\theta_m(x,t,\Omega)} \{ B h_m(x,\Omega) h_m^{*}(x,\Omega) - h_m(x,\Omega) h_m^{*}(x,\Omega) B \},
\]
where $\theta_m (x, t, \Omega) = 1 + (e^t - 1) \int_0^x |h_m (s, \Omega)|^2 ds$.
So, for arbitrary $t \in \mathbb{R}$, $\lambda_n(\Omega,t) = \lambda_n (\Omega)$ for all $n \in \mathbb{Z}$, $a_n(\Omega,t) = a_n (\Omega)$ for all $n \in \mathbb{Z} \backslash \{ m\}$ and $a_m(\Omega,t) = a_m (\Omega) e^{-t}$.
\end{theorem}

Theorem \ref{thm2.1} shows that it is possible to change exactly one norming constant, keeping the others.

Changing successively each $a_m (\Omega)$ by $a_m (\Omega) e^{-t_m}$, we can obtain any isospectral potential, corresponding to the sequence
$\{ t_m; m \in \mathbb{Z} \} \in l^2 $.

In \cite{Ha2} were used the following designations:

\begin{flushleft}
\quad \quad $T_{-1} = \{ \ldots, 0, \ldots \}$,\\
\quad \quad $T_{0} = \{ \ldots, 0, \ldots, 0, t_0,  0, \ldots, 0, \ldots \}$, \\
\quad \quad $T_{1} = \{ \ldots, 0, \ldots, 0, 0, t_0, t_1, 0, \ldots, 0, \ldots \}$, \\
\quad \quad $T_{2} = \{ \ldots, 0, \ldots, 0, t_{-1}, t_0, t_1, 0, \ldots, 0, \ldots \}$, \\
\quad \quad \ldots,\\
\quad \quad $T_{2n} = \{ \ldots, 0, 0, t_{-n}, \ldots, t_{-1}, t_0, t_1, \ldots, t_{n-1}, t_{n}, 0, \ldots \}$, \\
\quad \quad $T_{2n+1} = \{ \ldots, 0, t_{-n}, t_{-n+1}, \ldots, t_{-1}, t_0, t_1, \ldots, t_{n}, t_{n+1}, 0, \ldots \}$, \\
\quad \quad \ldots.
\end{flushleft}
Let $\Omega(x, T_{-1}) \equiv \Omega(x)$ and

\[
\Omega(x, T_{m}) = \Omega(x, T_{m-1}) + \bigtriangleup \Omega(x, T_{m}), \quad m = 0, 1, 2, \ldots,
\]
where

\[
\bigtriangleup \Omega(x, T_{m}) = \cfrac{e^{t_{\tilde{m}}} - 1}{\theta_m(x, t_{\tilde{m}}, \Omega(\cdot, T_{m-1}))}
                                  [ B h_{\tilde{m}}(x, \Omega(\cdot, T_{m-1})) h_{\tilde{m}}^{*}(\cdot)
                                  - h_{\tilde{m}}(\cdot) h_{\tilde{m}}^{*}(\cdot) B ],
\]
where $\tilde{m} = \cfrac{m+1}{2}$, if $m$ is odd and $\tilde{m} = - \cfrac{m}{2}$, if $m$ is even.
The arguments in others $h_{\tilde{m}}(\cdot)$ and $h_{\tilde{m}}^{*}(\cdot)$ are the same as in the first.
And after that in \cite{Ha2} was proved:

\begin{theorem}\label{thm2.2}{\cite{Ha2}.}
Let $T = \{ t_n, n \in \mathbb{Z} \} \in l^2 $ and $\Omega \in L^2_{\mathbb{R}}[0, \pi]$. Then

\begin{equation*}
\Omega(x, T) \equiv \Omega(x) + \sum_{m=0}^{\infty} \bigtriangleup \Omega(x, T_{m}) \in M^2(\Omega).
\end{equation*}
\end{theorem}

We see, that each potential matrix $\bigtriangleup \Omega(x, T_{m})$ defined by normalized eigenfunctions $h_{\tilde{m}}(x, \Omega(x, T_{m-1}))$ of the previous operator $L(\Omega(\cdot, T_{m-1}), \alpha, 0)$.
This approach we call "recurrent" description.

If we denote
\[
\cfrac{\partial \lambda_n}{\partial \Omega(x)} :=
            \left(
             \begin{array}{cc}
               \cfrac{\partial \lambda_n}{\partial p(x)} & \cfrac{\partial \lambda_n}{\partial q(x)} \\
               \cfrac{\partial \lambda_n}{\partial q(x)} & -\cfrac{\partial \lambda_n}{\partial p(x)} \\
             \end{array}
           \right)
           =
           \left(
             \begin{array}{cc}
               h_{n_1}^2(x) - h_{n_2}^2(x)  & 2 h_{n_1}(x)  h_{n_2}(x)  \\
                & \\
               2 h_{n_1}(x)  h_{n_2}(x)  & -(h_{n_1}^2(x)  - h_{n_2}^2(x) ) \\
             \end{array}
           \right),
\]
we will have
\begin{equation}\label{eq2.3}
  B \cfrac{\partial \lambda_n}{\partial \Omega(x) }
           =
           \left(
             \begin{array}{cc}
             2 h_{n_1}(x)  h_{n_2}(x)  & h_{n_2}^2(x)  - h_{n_1}^2(x)  \\
             h_{n_2}^2(x)  - h_{n_1}^2(x)  & -2 h_{n_1}(x)  h_{n_2}(x)  \\
             \end{array}
           \right).
\end{equation}
And it is easy to see, that the term
$[ B h_{\tilde{m}}(x, \Omega(\cdot, T_{m-1})) h_{\tilde{m}}^{*}(\cdot)- h_{\tilde{m}}(\cdot) h_{\tilde{m}}^{*}(\cdot) B ]$
of $\bigtriangleup \Omega(x, T_{m})$ is equal to $B \cfrac{\partial \lambda_n}{\partial \Omega(x, T_{m})}$.
Therefore the Theorems \ref{thm2.1} and \ref{thm2.2} can be rewritten as

\begin{theorem}\label{thm2.3}
Let $t \in \mathbb{R}$, $\alpha \in \Big( - \frac{\pi}{2}, \frac{\pi}{2} \Big]$. Then

\[
\Omega(x,t) = \Omega(x) + \cfrac{(e^t - 1)}{\theta_m(x,t,\Omega)} B \cfrac{\partial \lambda_m}{\partial \Omega(\cdot, T_{m}))},
\]
where $\theta_m (x, t, \Omega) = 1 + (e^t - 1) \int_0^x |h_m (s, \Omega)|^2 ds$.
So, for arbitrary $t \in \mathbb{R}$, $\lambda_n(\Omega,t) = \lambda_n (\Omega)$ for all $n \in \mathbb{Z}$, $a_n(\Omega,t) = a_n (\Omega)$ for all $n \in \mathbb{Z} \backslash \{ m\}$ and $a_m(\Omega,t) = a_m (\Omega) e^{-t}$.
\end{theorem}

\begin{theorem}\label{thm2.4}
Let $T = \{ t_n, n \in \mathbb{Z} \} \in l^2 $ and $\Omega \in L^2_{\mathbb{R}}[0, \pi]$. Then

\[
\Omega(x, T) \equiv \Omega(x) +
\sum_{m=0}^{\infty} \cfrac{e^{t_{\tilde{m}}} - 1}{\theta_m(x, t_{\tilde{m}}, \Omega(x, T_{m-1}))} B \cfrac{\partial \lambda_{\tilde{m}}}{\partial \Omega(x, T_{m-1}))}.
\]
\end{theorem}

\section{\large{Changing spectral data on half axis.}}\label{sec3}

Let us consider canonical Dirac operator $L(p, q, \alpha)$ on $[0,\infty)$, which has a pure discrete spectra.
In work \cite{Ha3}, Harutyunyan proved, that in this case one can add or subtract a finite number of eigenvalues, or scale the values of norming constants (i.e. $a_n$ to change by $e^t a_n $, for arbitrary $t \in \mathbb{R}$).
In that work explicit formulae for potential functions of changed operator are given.

According to the paper \cite{Ha3}, when we want to add a new eigenvalue $\mu$, the formula for potential function $\Omega_1(x)$ will be:

\begin{equation}\label{eq3.1}
\begin{array}{c}
\Omega_1(x) \equiv \Omega(x) + \cfrac{1}{1 + \int_{0}^{x} |h(t, \mu)|^2 dt} 
\{ B h(x,\mu) h^{*}(x,\mu) - \qquad \qquad \\ 
\\
\qquad \qquad \qquad \qquad \qquad \qquad - h(x,\mu) h^{*}(x,\mu) B \}.
\end{array}
\end{equation}

When we want to subtract an eigenvalue, e.g. $\lambda_0$, the formula for potential function $\Omega_2(x)$ will be:

\begin{equation}\label{eq3.2}
\begin{array}{c}
\Omega_2(x) \equiv \Omega(x) - \cfrac{1}{1 - \int_{0}^{x} |h(t, \lambda_0)|^2 dt} 
\{ B h(x,\lambda_0) h^{*}(x,\lambda_0) - \qquad \qquad\\ 
\\
\qquad \qquad \qquad \qquad \qquad \qquad - h(x,\lambda_0) h^{*}(x,\lambda_0) B \}.
\end{array}
\end{equation}

When we want to scale the value of a norming constant, e.g. $a_0$, which corresponds to eigenvalue $\lambda_0$, the formula for potential function $\Omega_3(x)$ will be:

\begin{equation}\label{eq3.3}
\begin{array}{c}
\Omega_3(x) \equiv \Omega(x) + \cfrac{e^{-t} - 1}{1 + (e^{-t} - 1) \int_{0}^{x} |h(t, \lambda_0)|^2 dt} 
\{ B h(x,\lambda_0) h^{*}(x,\lambda_0) -  \\ 
\\
\qquad \qquad \qquad \qquad \qquad  - h(x,\lambda_0) h^{*}(x,\lambda_0) B \}.
\end{array}
\end{equation}
 
Using formula \eqref{eq2.3} we can rewrite the formulae \eqref{eq3.1}--\eqref{eq3.3} in terms of eigenvalues' gradient:

\begin{equation}\label{eq3.4}
\Omega_1(x) \equiv \Omega(x) + \cfrac{1}{1 + \int_{0}^{x} |h(t, \mu)|^2 dt} 
\cdot
\cfrac{\partial \mu}{\partial \Omega(x)},
\end{equation}
\begin{equation}\label{eq3.5}
\Omega_2(x) \equiv \Omega(x) - \cfrac{1}{1 - \int_{0}^{x} |h(t, \lambda_0)|^2 dt}
\cdot  
\cfrac{\partial \lambda_0}{\partial \Omega(x)},
\end{equation}
\begin{equation}\label{eq3.6}
\Omega_3(x) \equiv \Omega(x) + \cfrac{e^{-t} - 1}{1 + (e^{-t} - 1) \int_{0}^{x} |h(t, \lambda_0)|^2 dt} 
\cdot
\cfrac{\partial \lambda_0}{\partial \Omega(x)}.
\end{equation}

In \cite{Ha3} there is also given a formula for changing finite number of eigenvalues or norming constants.
If we want to add $n$ number of eigenvalues $\mu_k$, to subtract $m$ number of eigenvalues $\lambda_k$ and to scale $l$ number of norming constants $a_k$, then the formula for such potential $\tilde{\Omega}(x)$ depending of initial potential $\Omega(x)$ will be:
\begin{equation}\label{eq3.7}
\begin{array}{c}
\tilde{\Omega}(x) \equiv \Omega(x) + \displaystyle \sum_{k=1}^{n+m+l}
\cfrac{\gamma_k}{1 + \gamma_k \int_{0}^{x} |h(t,\Omega_{k-1}, \nu_k)|^2 dt} \cdot \\
\cdot \{ B h(x, \Omega_{k-1}, \nu_k) h^{*}(x, \Omega_{k-1}, \nu_k) - h(x, \Omega_{k-1}, \nu_k) h^{*}(x, \Omega_{k-1}, \nu_k) B \}.
\end{array}
\end{equation}
where
\begin{equation*}
\gamma_k = 
\Bigg \{ \begin{array}{c}
1, \qquad \qquad \qquad 1 \leq k \leq n, \\

\qquad -1, \qquad \qquad n+1 \leq k \leq n+m, \\

\qquad e^{-t} -1,\qquad  n+m+1 \leq k \leq n+m+l,
\end{array}
\end{equation*}
\begin{equation*}
\nu_k = 
\Bigg \{ \begin{array}{c}
\mu_k, \qquad \qquad \qquad 1 \leq k \leq n, \\
\qquad \qquad \lambda_k, \qquad \qquad n+1 \leq k \leq n+m+l,
\end{array}
\end{equation*}
and potential function $\Omega_0(x) = \Omega(x)$ and  $\Omega_k(x)$, for $k = 0, 1, \ldots, n+m+l$, are given by formula:
\begin{equation*}
\begin{array}{c}
\Omega_k(x) = \Omega_{k-1}(x) +
\cfrac{\gamma_k}{1 + \gamma_k \int_{0}^{x} |h(t,\Omega_{k-1}, \nu_k)|^2 dt} \cdot \\
\cdot \{ B h(x, \Omega_{k-1}, \nu_k) h^{*}(x, \Omega_{k-1}, \nu_k) - h(x, \Omega_{k-1}, \nu_k) h^{*}(x, \Omega_{k-1}, \nu_k) B \}.
\end{array} 
\end{equation*}

Using formula \eqref{eq2.3} we can rewrite the \eqref{eq3.7} in terms of eigenvalues' gradient:
\begin{equation}\label{eq3.8}
\tilde{\Omega}(x) \equiv \Omega(x) + \displaystyle \sum_{k=1}^{n+m+l}
\cfrac{\gamma_k}{1 + \gamma_k \int_{0}^{x} |h(t,\Omega_{k-1}, \nu_k)|^2 dt} \cdot
\cfrac{\partial \nu_k}{\partial \Omega_{k-1}(x)}
\end{equation}

\textbf{Acknowledgment.} This work was supported by the RA MES State Committee of Science, in the frames of the research project No.15T-1A392.

\end{document}